\documentclass{amsart}
\usepackage{mathrsfs,amssymb,multirow,setspace}

\theoremstyle{plain}
\newtheorem{theorem}{Theorem}[section]
\newtheorem{lemma}[theorem]{Lemma}
\newtheorem{proposition}[theorem]{Proposition}
\newtheorem{corollary}[theorem]{Corollary}

\theoremstyle{definition}
\newtheorem{notation}[theorem]{Notation}
\newtheorem{example}[theorem]{Example}
\newtheorem{definition}[theorem]{Definition}

\theoremstyle{remark}

\input xy
\xyoption{all}

\def\g{\gamma}
\def\G{\Gamma}
\def\vt{\vartheta}

\newcommand{\sis}[2]{^{#1}\!\zeta_{#2}}
\begin{document}


\title[Rank properties of $A^+(B_n)$]{Rank Properties of Multiplicative Semigroup Reduct of Affine Near-Semirings over $B_n$}

\author[Jitender Kumar, K. V. Krishna]{Jitender Kumar and  K. V. Krishna}
\address{Department of Mathematics, Indian Institute of Technology Guwahati, Guwahati, India}
\email{\{jitender, kvk\}@iitg.ac.in}


\begin{abstract}
This work investigates the rank properties of $A^+(B_n)$, the multiplicative semigroup reduct of the affine near-semirings over an aperiodic Brandt semigroup $B_n$. In this connection, the work obtains the small rank, lower rank and large rank of $A^+(B_n)$. Further, the work provides lower bounds for intermediate rank and upper rank of $A^+(B_n)$.
\end{abstract}



\maketitle

\section*{Introduction}

Since the work of Marczewski in \cite{a.marczewski66}, many authors have studied the rank properties in the context of general algebras. The concept of rank for general algebras is analogous to the concept of dimension in linear algebra. The dimension of a vector space is the maximum cardinality of an independent subset, or equivalently, it is the minimum cardinality of a generating set of the vector space.
A subset $U$ of a semigroup $\G$ is said to be \emph{independent} if every element of $U$ is not in the subsemigroup generated by the remaining elements of $U$, i.e. \[ \forall a \in U, \; a \notin \langle U \setminus \{a\} \rangle .\] This definition of independence is analogous to the usual definition of independence in linear algebra. It can be observed that the minimum size of a generating set need not be equal to the maximum size of an  independent set in a semigroup. Accordingly, Howie and Ribeiro have considered the following possible definitions of ranks for a finite semigroup $\G$ (cf. \cite{a.Howie99,a.Howie00}).
\begin{enumerate}
\item $r_1(\G) = \max\{k :$ every subset $U$ of cardinality $k$ in $\G$ is independent\}.
\item $r_2(\G) = \min\{|U| : U \subseteq \G, \langle U\rangle = \G\}$.
\item $r_3(\G) = \max\{|U| : U \subseteq \G, \langle U\rangle = \G, U$ is independent\}.
\item $r_4(\G) = \max\{|U| : U \subseteq \G, U$  is independent\}.
\item $r_5(\G) = \min\{k :$ every subset $U$ of cardinality $k$ in $\G$ generates $\G$\}.
\end{enumerate}
It can be observed that \[r_1(\G) \le r_2(\G) \le r_3(\G) \le r_4(\G) \le r_5(\G).\] Thus,
$r_1(\G), r_2(\G), r_3(\G), r_4(\G)$ and $r_5(\G)$ are, respectively, known as \emph{small rank}, \emph{lower rank}, \emph{intermediate rank}, \emph{upper rank} and \emph{large rank} of $\G$.

While all these five ranks coincide for certain semigroups, there exist semigroups for which all these ranks are distinct. For instance, all these five ranks are equal to $|\G|$ for a finite left or right zero semigroup $\G$.  For $n > 2$,
Howie et al. have determined all these five ranks for an aperiodic Brandt semigroup $B_n$ through the papers \cite{a.Howie87,a.Howie99,a.Howie00} and observed that all these five ranks are different from each other.

The ranks of rectangular bands and monogenic semigroups were also established in \cite{a.Howie99,a.Howie00}. The lower rank of completely 0-simple semigroups was obtained by Ru\v{s}kuc \cite{a.ruskuc94}. The intermediate rank of $S_n$, the symmetric group of degree $n$, is determined to be $n-1$ by Whiston \cite{a.whiston00}. All independent generating sets of size $n-1$ in $S_n$ were investigated in \cite{a.camron02}. In \cite{t.jdm02}, Mitchell studied the rank properties of various groups, semigroups and semilattices. The rank properties of certain semigroups of order preserving transformations have been investigated in \cite{a.howie92} and further extended to orientation-preserving transformations in \cite{a.ping11}. Jitender and Krishna have studied the ranks of additive semigroup reduct of affine near-semirings over an aperiodic Brandt semigroup $B_n$ in \cite{a.jk13-2}.

In this work, we investigate the rank properties of $A^+(B_n)$ -- the multiplicative semigroup reduct of the affine near-semiring over an aperiodic Brandt semigroup $B_n$.  In this connection, we obtain the small, lower and large ranks of $A^+(B_n)$. Further, we provide lower bounds for intermediate rank and upper rank of $A^+(B_n)$. The remaining paper has been organized into six sections. Section 1 provides a necessary background material for the subsequent four main sections which are devoted to $r_1, r_2, r_3, r_4$ and $r_5$ of $A^+(B_n)$. We conclude the paper in Section 6.

\section{Preliminaries}

In this section, we provide a necessary background material and fix our notation. For more details one may refer to \cite{a.jk13}.

\begin{definition}
An algebraic structure $(S, +, \cdot)$ is said to be a \emph{near-semiring} if
\begin{enumerate}
\item $(S, +)$ is a semigroup,
\item $(S, \cdot)$ is a semigroup, and
\item $a\cdot(b + c) = a\cdot b + a\cdot c$, for all $a,b,c \in S$.
\end{enumerate}
\end{definition}

In this work, unless it is required, algebraic structures (such as semigroups, groups, near-semirings) will simply be referred by their underlying sets without explicit mention of their operations. Further, we write an argument of a function on its left, e.g. $xf$ is the value of a function $f$ at an argument $x$.

\begin{example}
Let $(\Gamma, +)$ be a semigroup and $M(\Gamma)$ be the set of all
mappings on $\Gamma$. The algebraic structure $(M(\G), +, \circ)$ is a
near-semiring, where $+$ is point-wise addition and $\circ$ is composition
of mappings, i.e., for $\gamma \in \Gamma$ and $f,g \in M(\Gamma)$,
$$\g(f + g)= \g f + \g g \;\;\;\; \text{and}\;\;\;\; \g(f \circ g) = (\g
f)g.$$ Also, certain subsets of $M(\G)$ are near-semirings. For instance,
the set $M_c(\Gamma)$ of all constant mappings on $\Gamma$ is a
near-semiring with respect to the above operations so that $M_c(\Gamma)$
is a subnear-semiring of $M(\G)$.
\end{example}

Now, we recall the notion of affine near-semirings from \cite{t.kvk05a}. Let $(\G, +)$ be a semigroup. An element $f \in M(\G)$ is said to be an \emph{affine map} if $f = g + h$, for some $g \in End(\G)$, the set of all endomorphisms over $\G$, and $h \in M_c(\G)$. The set of all affine mappings over $\G$, denoted by $\text{Aff}(\G)$, need not be a subnear-semiring of $M(\G)$. The \emph{affine near-semiring}, denoted by $A^+(\G)$, is the subnear-semiring generated by $\text{Aff}(\G)$ in $M(\G)$. Indeed, the subsemigroup of $(M(\G), +)$ generated by $\text{Aff}(\G)$ equals $(A^+(\G), +)$ (cf. \cite[Corollary 1]{a.kvk05}). If $(\G, +)$ is commutative, then $\text{Aff}(\G)$ is a subnear-semiring of $M(\G)$ so that $\text{Aff}(\G) = A^+(\G)$.

\begin{definition}
Given a finite group $G$ and a natural number $n$, write $[n] = \{1,2, \ldots, n\}$ and $B(G, n) = ([n] \times G \times [n]) \cup \{\vt\}$.  Define a binary operation (say, addition) on $B(G, n)$ by
\[ (i, a, j)+(k, b, l) =
                \left\{\begin{array}{cl}
                (i, ab, l) & \text {if $j = k$;}  \\
                \vt     & \text {otherwise,}
                \end{array}\right. \]
\[\mbox{and }\; \vt + (i, a, j)  = (i, a, j) + \vt = \vt + \vt =\vt.\]
With the above defined addition, $B(G, n)$ is a semigroup known as the \emph{Brandt semigroup}.
When $G$ is the trivial group, the Brandt semigroup $B(\{e\}, n)$ is aperiodic and it is denoted by $B_n$. For more details on Brandt semigroups, one may refer to \cite{b.Howie95}.
\end{definition}

In \cite{a.jk13}, Jitender and Krishna have studied the structure of (both additive and multiplicative) semigroup reducts of the near-semiring $A^+(B_n)$.  We now recall the results on $A^+(B_n)$ which are useful in the present work. The following concept plays a vital role in the study of $A^+(B_n)$.

Let $(\Gamma, +)$ be a semigroup with zero element $\vartheta$. For  $f \in M(\Gamma)$, the \emph{support of $f$}, denoted by supp$(f)$, is
defined by the set \[ {\rm supp}(f) = \{\alpha \in \Gamma \;|\; \alpha f \neq \vartheta\}.\] A function $f \in M(\Gamma)$ is said to be of \emph{k-support} if the cardinality of supp$(f)$ is $k$, i.e. $|{\rm supp}(f)| = k$. If $k = |\Gamma|$ (or $k = 1$), then $f$ is said to be of \emph{full support} (or \emph{singleton support}, respectively). For $X \subseteq M(\Gamma)$, we write $X_k$ to denote the set of all mappings of $k$-support in $X$, i.e. $$ X_k = \{ f \in X \mid f \; \text{is of $k$-support}\;  \}.$$

For ease of reference,  we continue to use the following notations for the elements of $M(B_n)$, as given in \cite{a.jk13}.

\begin{notation}\

\begin{enumerate}
\item For $c \in B_n$, the constant map that sends all the elements of $B_n$ to $c$ is denoted by $\xi_c$. The set of all constant maps over $B_n$ is denoted by $\mathcal{C}_{B_n}$.
\item For $k, l, p, q \in [n]$, the singleton support map that maps $(k, l)$ to $(p, q)$ is denoted by $\sis{(k, l)}{(p, q)}$.
\item For $p, q \in [n]$, the $n$-support map which sends $(i, p)$ (where $1 \le i \le n$) to $(i\sigma, q)$ using a permutation $\sigma \in S_n$ is denoted by $(p, q; \sigma)$. We denote the identity permutation on $[n]$ by $id$.
\end{enumerate}
\end{notation}

\begin{theorem}[\cite{a.jk13}]\label{t.affbn}
For $n \geq 1$, ${\rm Aff}(B_n) = {\rm Aff}(B_n)_n \cup \mathcal{C}_{B_n}$, where ${\rm Aff}(B_n)_n = \{(p, q; \sigma)\; |\; p, q \in [n], \sigma \in S_n\}$. Hence, $|{\rm Aff}(B_n)| = (n! + 1)n^2  + 1$.
\end{theorem}

Note that $A^+(B_1) = \{(1, 1; id)\} \cup \mathcal{C}_{B_1}$. For $n \ge 2$, the elements of $A^+(B_n)$ are given by the following theorem.

\begin{theorem}[\cite{a.jk13}]\label{t.class.a+bn}
For $n \geq 2$, $A^+(B_n) = {\rm Aff}(B_n) \cup \left\{\sis{(k, l)}{p, q}\; |\; k, l, p, q \in [n]\right\}$. Hence, $|A^+(B_n)| = (n! + 1)n^2 + n^4 + 1$.
\end{theorem}

In what follows, $A^+(B_n)$ denotes the multiplicative semigroup reduct $(A^+(B_n), \circ)$ of the affine near-semiring $(A^+(B_n), +, \circ)$.

\section{Small rank}

It can be easily observed that $A^+(B_1)$ is an independent set and none of its proper subsets generate $A^+(B_1)$. Hence, for $1 \le i \le 5$, we have \[r_i(A^+(B_1))  = |A^+(B_1)| = 3.\] In the rest of the paper we shall investigate the ranks of $A^+(B_n)$, for $n > 1$.  We obtain the small rank of $A^+(B_n)$ as a consequence of the following result due to Howie and Ribeiro.

\begin{theorem}[\cite{a.Howie00}]\label{r1-gm}
Let $\G$ be a finite semigroup, with $|\G| \ge 2$. If $\G$ is not a band, then $r_1(\G) = 1$.
\end{theorem}

Owing to the fact that $A^+(B_n)$, for $n \ge 2$, have some non idempotent elements, it is not a band. For instance, the singleton support maps $\sis{(k, l)}{(p, q)}$ with $(k, l) \ne (p, q)$ in $A^+(B_n)$ are not idempotents. Hence, we have the following corollory of Theorem \ref{r1-gm}.

\begin{corollary}
For $n \ge 2$, $r_1(A^+(B_n)) = 1$.
\end{corollary}

\section{Lower rank}

In this section, first we ascertain that the set of $n$-support elements of $A^+(B_n)$ along with $\xi_\vt$ forms a subsemigroup which is isomorphic to the Brandt semigroup $B(S_n, n)$. Using this key result we obtain the lower rank of the semigroup $A^+(B_n)$. In what follows, a generating set of minimum cardinality is termed as a minimum generating set.

\begin{lemma}\label{l.a+bnc-pro}
For $n \ge 2$, let $f,  g_i (1 \le i \le k) \in A^+(B_n) \setminus \{\xi_\vt\}$ such that $f = g_1g_2 \cdots g_k$. Then,
\begin{enumerate}
\item $f \in A^+(B_n)_{n^2 + 1}$ if and only if $g_j \in  A^+(B_n)_{n^2 + 1}$ for some $j$;
\item $f \in A^+(B_n)_n$ if and only if $g_i \in  A^+(B_n)_n$ for all $i$; and
\item if $f \in A^+(B_n)_1$ then $g_j \in  A^+(B_n)_1$ for some $j$.
\end{enumerate}
\end{lemma}

\begin{proof}$\;$
\begin{enumerate}
\item If $g_i \notin  A^+(B_n)_{n^2 + 1}$ for all $i$, then, by \cite[Proposition 2.7]{a.jk13}, $\vt \notin {\rm supp}(g_i)$ so that  $\vt \notin {\rm supp}(f)$ and hence $f \not\in A^+(B_n)_{n^2 + 1}$. Since the composition of a constant map with any map is a constant map, we have the converse.
\item  Suppose $f \in A^+(B_n)_n$. From (1) above, $g_i \notin \mathcal{C}_{B_n}$ for all $i$. If $g_j \in A^+(B_n)_1$, for some $j$, then clearly $|{\rm supp}(f)| \le 1$.  Hence, by Theorem \ref{t.class.a+bn}, we have $g_i \in  A^+(B_n)_n$ for all $i$. Conversely, for $1 \le i \le k$, suppose $g_i = (p_i, q_i; \sigma_i)$. For $1 \le i \le k-1$, note that  $g_ig_{i + 1}$ is either $\xi_\vt$  or $(p_i, q_{i + 1}; \sigma_i\sigma_{i + 1})$, where $q_i = p_{i + 1}$. Consequently, since $f \ne \xi_{\vt}$, we have $f = (p_1, q_k; \sigma_1\sigma_2\cdots \sigma_k) \in A^+(B_n)_n$.
\item Follows from (1) and (2).
\end{enumerate}
\end{proof}

\begin{corollary}\label{c.sis-fs-gset}
Any generating  subset of $A^+(B_n)$ contains at least a singleton support element and a full support element.
\end{corollary}

In view of Lemma \ref{l.a+bnc-pro}(2), $A^+(B_n)_n \cup \{\xi_\vt\}$ is a subsemigroup of $A^+(B_n)$. Further, we have the following lemma regarding $A^+(B_n)_n \cup \{\xi_\vt\}$.

\begin{lemma}\label{l.iso.bgn-ns}
The semigroup $A^+(B_n)_n \cup \{\xi_\vt\}$ is isomorphic to the semigroup $(B(S_n, n), +)$.
\end{lemma}

\begin{proof}
Note that the assignment $(i, j; \sigma) \mapsto (i, \sigma, j)$ and $\xi_\vt \mapsto \vt$, for all $i,j \in [n]$ and $\sigma \in S_n$, is an isomorphism.
\end{proof}

\begin{lemma}\label{l.gen.nsupp}
Let $\sigma$ be the cycle $(1\; 2 \cdots \; n)$ and $\tau$ be the transposition $(1\;  2)$ in $S_n$. The following are minimum generating subsets of the semigroups $A^+(B_n)_n \cup \{\xi_\vt\}$, for $n \ge 2$.
\begin{enumerate}
\item If $n \ge 3$, $\mathcal{P} = \{(1, 1; \sigma), (1, 2; \tau), (2, 3; id), \cdots (n - 1, n; id), (n, 1; id)\}$.
\item If $n = 2$, $\mathcal{P'} =\{(1, 2; \sigma), (2, 1; id)\}$.
\end{enumerate}
\end{lemma}

\begin{proof}
Given a minimum generating set $\{g_1, \ldots, g_r\}$ of a finite group $G$ with the identity element $e$, by \cite[Proposition 2.4]{a.garba94},
the set \[ \{(1, g_1, 1), \ldots(1, g_{r-1}, 1), (1, g_r, 2),(2, e, 3 ), \ldots, (n -1, e, n), (n, e, 1)\}\]  of $r + n - 1$ elements is a minimum generating set of Brandt semigroup $B(G, n)$. For $n \ge 2$, since $\{\sigma, \tau\}$ is a minimum generating subset of the symmetric group $S_n$, by Lemma \ref{l.iso.bgn-ns}, the result follows.
\end{proof}

\begin{theorem}\label{t.r2-a+bnc}
For $n \ge 3$, $r_2(A^+(B_n))  =  n + 3$.
\end{theorem}

\begin{proof}
We prove that \[ \mathcal{Q} = \mathcal{P} \cup \left\{\xi_{(1, 1)},\; \sis{(1, 1)}{(1, 1)}\right\}\] is a minimum generating set of $A^+(B_n)$ so that the result follows.

By Lemma \ref{l.gen.nsupp}(1), $\mathcal{Q}$ generates all $n$-support maps and the zero map in $A^+(B_n)$.  For $f \in A^+(B_n)$, if $f =$ $\sis{(k, l)}{(p, q)}$, then write \[\sis{(k, l)}{(p, q)} = (l, 1; \sigma)\sis{(1, 1)}{(1, 1)}(1, q; \rho)\] or if $f = \xi_{(p, q)}$, then write \[\xi_{(p, q)} = \xi_{(1, 1)}(1, q; \rho),\]  where $k \sigma = 1$ and $1\rho = p$,  so that $f \in \langle \mathcal{Q} \rangle$. Hence, by Theorem \ref{t.class.a+bn}, we have $\langle\mathcal{Q} \rangle =  A^+(B_n)$.

Let $V$ be a generating subset of $A^+(B_n)$. By Lemma \ref{l.a+bnc-pro}(2) and Lemma \ref{l.gen.nsupp}, $V$ must contain at least $n + 1$ elements of $n$-support to generate all $n$-support elements in $A^+(B_n)$.  Further, by Corollary \ref{c.sis-fs-gset}, $V$ contains at least a singleton support element and a full support element so that $|V| \ge n + 3$. Hence, the result.
\end{proof}

\begin{theorem}\label{t.r2-a+b2c}
$r_2(A^+(B_2))  = 4$.
\end{theorem}

\begin{proof}
In the similar lines of the proof of Theorem \ref{t.r2-a+bnc}, note that the set $\mathcal{P'} \cup \left\{\xi_{(1, 1)},\; \sis{(1, 1)}{(1, 1)}\right\}$ is a minimum generating set of the semigroup $A^+(B_2)$.
\end{proof}

\section{Intermediate and upper rank}

In this section, we will only provide lower bounds for intermediate and upper ranks of $A^+(B_n)$. In view of Lemma \ref{l.iso.bgn-ns}, we shall rely on some known lower bounds of respective ranks for $B(G, n)$. First we recall the required results and proceed to give the lower bounds in theorems \ref{t.l-int-rank} and \ref{t.l-upp-rank}.

\begin{theorem}[\cite{p.jdm05}]\label{t.ind-gen-Bgn}
Let $X$ be an independent generating set of maximum cardinality in a finite group $G$ with identity element $e$  and $\{I, J\}$ be a partition of the set $[n]$ such that $|I| = \left\lceil n /2 \right\rceil \ $  and $|J| = \left\lfloor n /2 \right\rfloor$. Then, in $B(G, n)$,
\begin{enumerate}
\item $\{(2, e, 3), (3, e, 4), \ldots, (n-1, e, n), (n, e, 1)\} \cup \{(1, x, 2)\; |\; x \in X\}$ is an independent generating set, and
\item $\{(i, e, i)\; |\; i \in [n]\} \cup \{(i, g, j)\; |\; i \in I, j \in J, g \in G\}$ is an independent set.
\end{enumerate}
\end{theorem}

\begin{theorem}[\cite{a.whiston02}]
The set $\mathcal{T} = \{(1\;  2),(2 \; 3), \cdots, (n-1 \;  n)\}$ of transpositions is an independent generating set of maximum cardinality in $S_n$ and hence $r_3(S_n) = n - 1$.
\end{theorem}

Now we prove the following lemma regarding an independent generating subset of $A^+(B_n)$.

\begin{lemma}\label{l.ind-gen}
Any independent generating subset of $A^+(B_n)$ contains
\begin{enumerate}
\item exactly one singleton support element, and
\item exactly one full support element.
\end{enumerate}
\end{lemma}

\begin{proof} In view of Corollary \ref{c.sis-fs-gset}, let $f$ =  $\sis{(k, l)}{(p, q)}$ and $g = \xi_{(p, q)}$ are in $U$.
\begin{enumerate}
\item Suppose there is another singleton support map, say $f'$ = $\sis{(s, t)}{(u, v)}  \in U$. Consider the $n$-support maps $h = (t, l;  \sigma)$ with $s \sigma = k$ and $h' = (q, v;  \tau)$ with $p \tau = u$. Since $U$ is a generating set, we have $h, h' \in \langle U \setminus \{f'\} \rangle$.  Now observe that $f' = hfh'$ so that $f'  \in \langle U \setminus \{f'\} \rangle$; a contradiction to $U$ is an independent set.
\item Suppose there is another full support map, say $g'$ = $\xi_{(u, v)}  \in U$. Consider the $n$-support map $h' = (q, v;  \tau)$ with $p \tau = u$ and note that  $h' \in \langle U \setminus \{g'\} \rangle$ (cf. Lemma \ref{l.a+bnc-pro}(2)). However, since $g' = gh'$, we have $g'  \in \langle U \setminus \{g'\} \rangle$; a contradiction to $U$ is an independent set.
\end{enumerate}
\end{proof}

\begin{theorem}\label{t.l-int-rank}
For $n \ge 2$, $r_3(A^+(B_n)) \ge 2n$.
\end{theorem}

\begin{proof}
We observe that the set  $\mathcal{X} = \mathcal{X}' \cup \left\{\xi_{(1 , 1)},\; \sis{(1, 1)}{(1, 1)} \right\}$, where
\[\mathcal{X}' = \{(2, 3; id), (3, 4; id), \ldots, (n-1, n; id), (n, 1; id)\} \cup \{(1, 2; \sigma)\; |\; \sigma \in \mathcal{T}\}\]
is an independent generating set in $A^+(B_n)$ so that $r_3(A^+(B_n)) \ge |\mathcal{X}| = 2n$. By Theorem \ref{t.ind-gen-Bgn}(1) and Lemma \ref{l.iso.bgn-ns}, the set $\mathcal{X}'$ generates all $n$-support maps and the zero map in $A^+(B_n)$. Now, in the similar lines of proof of Theorem \ref{t.r2-a+bnc}, one can prove that $\langle \mathcal{X} \rangle = A^+(B_n)$. Further, in view of Lemma \ref{l.ind-gen} and Theorem \ref{t.ind-gen-Bgn}(1), $\mathcal{X}$ is an independent subset in $A^+(B_n)$.
\end{proof}

Though Theorem \ref{t.l-int-rank} gives us a lower bound for upper rank of $A^+(B_n)$, in the following theorem we provide a better lower bound for $r_4(A^+(B_n))$.

\begin{theorem}\label{t.l-upp-rank}
For $n \ge 2$, $r_4(A^+(B_n)) \ge n! \left\lfloor n^2 / 4 \right\rfloor \ + n + 2 $.
\end{theorem}

\begin{proof}
Using Theorem \ref{t.ind-gen-Bgn}(2), Lemma \ref{l.iso.bgn-ns} and Lemma \ref{l.ind-gen}, one can observe that the set

\[\{(i, i; id)\; |\; i \in [n]\} \cup \{(i, j; \sigma)\; |\; i \in I, j \in J, \sigma \in S_n\} \cup \left \{\xi_{(1 , 1)},\; \sis{(1, 1)}{(1, 1)}  \right \},\] where $I$ and $J$ are as in Theorem \ref{t.ind-gen-Bgn}, is an independent subset in $A^+(B_n)$.
\end{proof}

\section{Large rank}

In this section, we obtain the large rank of $A^+(B_n)$. An element $a$ of a multiplicative semigroup $\G$ is said to be \emph{indecomposable} if there do not exist $b , c \in \G \setminus \{a\}$ such that $a = bc$. The following  key result by Howie and Ribeiro is useful to find the large rank of a finite semigroup.

\begin{theorem}[\cite{a.Howie00}]\label{r5-lsgp}
Let $\G$ be a finite semigroup and let $V$ be a proper subsemigroup of $\G$ with the largest possible size. Then $r_5(\G) = |V|+1.$ Hence, $r_5(\G) = |\G|$ if and only if $\G$ contains an indecomposable element.
\end{theorem}

\begin{proposition}\label{a+bnc-dec}
For $n \ge 2$, all the elements of $A^+(B_n)$ are decomposable.
\end{proposition}

\begin{proof}
Refereing to  Theorem \ref{t.class.a+bn}, we give a decomposition for each element $f \in A^+(B_n)$ in the following cases.
\begin{enumerate}
\item $f$ is the zero element: $\xi_\vt$ = $\sis{(k, l)}{(p, q)}$$\sis{(m, n)}{(s, t)}$, for$(p, q) \ne (m, n)$.
\item $f$ is of full support: Let $f  = \xi_{(p, q)}$. Then $\xi_{(p, q)} = \xi_{(m, n)}$ $\sis{(m, n)}{(p, q)}$, for $(m, n) \ne (p, q)$.
\item $f$ is of singleton support: Let $f$ = $\sis{(k, l)}{(p, q)}$. Then $f =$ $\sis{(k, l)}{(m, n)}$ $\sis{(m, n)}{(p, q)}$, for $(m, n) \not\in \{(p, q), (k, l)\}$.
\item $f$ is an $n$-support map: Let $f  = (k, p; \sigma)$. Note that $f = (k, q; \tau) (q, p; \tau^{-1}\sigma)$, for $q \ne p$ and $\tau \ne id$.
\end{enumerate}
\end{proof}

In order to find the large rank of $A^+(B_n)$, we adopt the technique that is introduced in \cite{a.jk13-3}. The technique, as stated in Lemma \ref{gm-lsgp}, relies on the concept of prime subsets of semigroups.  A nonempty subset $U$ of a (multiplicative) semigroup $\G$ is said to be \emph{prime} if, for all $a, b \in \G$, \[ab \in U\; \mbox{ implies }\; a \in U \; \mbox{ or }\; b \in U.\]

\begin{lemma}[\cite{a.jk13-3}]\label{gm-lsgp}
Let $V$ be a proper subset of a finite semigroup $\G$. Then $V$ is a
smallest prime subset of $\G$ if and only if $\G \setminus V$ is a largest subsemigroup of $\G$.
\end{lemma}

Using Lemma \ref{gm-lsgp} and Theorem \ref{r5-lsgp}, now we obtain the large rank of $A^+(B_n)$.

\begin{theorem}\label{r5-a+b2}
$r_5(A^+(B_2)) = 28.$
\end{theorem}

\begin{proof}
We show that the set $V = \{(1, 2; id), (1, 2; \sigma)\},$ where $\sigma$ is the cycle $(1\; 2)$ in $S_2$, is a smallest prime subset of $A^+(B_2)$ so that $r_5(A^+(B_2)) = 28$. Observe that
\[(1, 2; id) = (1, 1; id)(1, 2; id) = (1, 2; id)(2, 2; id) = (1, 1; \sigma)(1, 2; \sigma) = (1, 2; \sigma)(2, 2; \sigma)\]
and \[(1, 2; \sigma) = (1, 1; id)(1, 2; \sigma) = (1, 2; id)(2, 2; \sigma) = (1, 1; \sigma)(1, 2; id) = (1, 2; \sigma)(2, 2; id)\] are all the possible decompositions of $(1, 2; id)$ and $(1, 2; \sigma)$, respectively, with the elements of $A^+(B_2)$. Note that, every decomposition has at least one element from $V$ so that $V$ is a prime subset.

If $U$ is a prime subset of $A^+(B_2)$ with $|U| < |V|$, then $|U| = 1$. Hence, the element of $U$ will be indecomposable; a contradiction to Proposition \ref{a+bnc-dec}. Consequently, $V$ is a smallest prime subset of $A^+(B_2)$.
\end{proof}

\begin{theorem}\label{lr-a+bnc}
For $ n \geq 3,\,\, r_5(A^+(B_n)) = (n!)n^2+n^4+2.$
\end{theorem}
\begin{proof}
 We prove that $V = \{\xi_{(i, j)} \mid i, j \in [n] \}$ is a smallest prime subset of $A^+(B_n)$. Since $|V| = n^2$, the result follows from Theorem \ref{t.class.a+bn}.

By Lemma \ref{l.a+bnc-pro}(1), $V$ is a prime subset of $A^+(B_n)$. Let $U$ be a prime subset of $A^+(B_n)$ such that $|U|< |V|$. If $U \subset V$, then let $g = \xi_{(p, q)} \in V \setminus U$. Now, for $h = \xi_{(s, t)} \in U$, we have $h'$ =$\sis{(p, q)}{(s, t)}$  such that  $h = gh'$; which is not possible with $U$. Thus, $U \not \subset V$.

Let $f \in U \setminus V$. Then, $f$ can be (i) $\xi_\vt$, (ii)  a singleton support map, or (iii) an $n$-support map. In all the cases we observe that $|U| \ge n^2$, which is a contradiction to the choice of $U$.

(i) $f = \xi_\vt$: Note that $f = \xi_{(p, q)}$$\sis{(m, n)}{(p, q)}$, where $(p, q) \ne (m, n)$. For each pair $(p, q)$ we have one such decomposition of $f$ and hence there are at least $n^2$ such decompositions. From each decomposition one of the components, viz. a map with the image set $\{(p, q)\} \cup \{\vt\}$ is in $U$ so that $|U| \ge n^2$.

(ii) $f$ is singleton support map: Let $f$ = $\sis{(k, l)}{(p, q)}$. Note that  $f =$$\sis{(k, l)}{(m, n)}$$\sis{(m, n)}{(p, q)}$. For $(m, n) \ne (k, l)$, we have $n^2 - 1$ such decompositions of $f$. Being a prime subset, $U$ must contain at least one component from each decomposition so that $|U| \ge n^2-1$. Further, consider the decomposition \[\sis{(k, l)}{(p, q)} = \,\sis{(k, l)}{(k, l)}(l, q; \sigma),\] where $k \sigma = p$. One more element from the above decomposition should be in $U$. Thus, $|U| \ge n^2$.

(iii) $f$ is an $n$-support map: Let $f = (k, p; \sigma)$. Since $n \ge 3$, for fixed $q \not\in \{p, k\}$, consider the decomposition
\[(k, p; \sigma) = (k, q; \tau)(q, p; \tau^{-1}\sigma)\] one for each $\tau \in S_n$ so that there are
$n!$ such decompositions of $f$. Consequently, $|U| \ge n!$ so that $|U| > n^2$, for $n \ge 4$. For $n = 3$, in addition to above-mentioned six elements, we observe that there are another three elements in $U$. For instance, for each $\tau \in S_3$, none of the components in the decomposition \[(k, p; \sigma) = (k, k; \tau)(k, p; \tau^{-1}\sigma)\] are covered in any of the above-mentioned decompositions. Even if a left component reoccurs as a right component in any of the decompositions, at least  three of them will be in $U$. Hence, for $n \ge 3$, $|U| \ge  n^2$.
\end{proof}

\section{Conclusion}

In this work, we have investigated the ranks of $A^+(B_n)$, the multiplicative semigroup reduct of the affine near-semiring over an aperiodic Brandt semigroup, and obtained the small, lower and large ranks of $A^+(B_n)$, for all $n \ge 1$. While intermediate and upper ranks of $A^+(B_n)$ are not yet known, we have provided some lower bounds for these ranks.

In view of Theorem \ref{t.affbn}, the results embedded in this paper shall give us the respective ranks of the semigroup of affine transformations ${\rm Aff}(B_n)$. Note that $A^+(B_1) = {\rm Aff}(B_1)$ so that  $r_i({\rm Aff}(B_1)) = 3$, for $1 \le i \le 5$. For $n \ge 2$, observe that the $n$-support elements $(k,p; \sigma)$ with $p \ne k$ are not idempotent so that $r_1({\rm Aff}(B_n)) = 1$. While lower and large ranks of ${\rm Aff}(B_2)$ are 3 and 12 respectively, for $n \ge 3$, $r_2({\rm Aff}(B_n)) = n + 2$ and $r_5({\rm Aff}(B_n)) = (n!)n^2 + 2$.

\end{document}